\documentclass[10pt,a4paper]{article}

\usepackage{epsf,epsfig,amsfonts,amsgen,amsmath,amstext,amsbsy,amsopn,amsthm,amssymb, epstopdf
}
\usepackage{ebezier,eepic}
\usepackage{color}
\usepackage{multirow}
\usepackage{url}

\newtheorem{theorem}{Theorem}[section]

\newtheorem{lemma}{Lemma}
\newtheorem{false statement}{False statement}

\theoremstyle{definition}
\newtheorem{definition}{Definition}
\newtheorem{claim}{Claim}

\newtheorem{remark}[claim]{Remark}

\newtheorem{corollary}[claim]{Corollary}

\begin{document}

\title{\bf\Large Spanning trees in complete bipartite graphs and
resistance distance in nearly complete bipartite graphs}
\date{}

\author{Jun Ge\thanks{School of Mathematical Sciences,
Sichuan Normal University, Chengdu, 610066, Sichuan, P. R. China.
Email: mathsgejun@163.com}~ and Fengming Dong\thanks{Mathematics and Mathematics Education,
National Institute of Education,
Nanyang Technological University,
Singapore. Email:
fengming.dong@nie.edu.sg}}

\maketitle

\begin{abstract}
Using the theory of electrical network, we first obtain
a simple formula for the number of spanning trees of a complete bipartite graph
containing a certain matching or a certain tree.
Then we apply the effective resistance (i.e.,
resistance distance in graphs)
to find a formula for the number of spanning trees
in the nearly complete bipartite graph
$G(m,n,p)=K_{m,n}-pK_2$ $(p\leq \min\{m,n\})$,
which extends a recent result by
Ye and Yan who obtained the effective resistances and the
number of spanning trees in $G(n,n,p)$.
As a corollary, we obtain the Kirchhoff index of $G(m,n,p)$
which extends a previous result by Shi and Chen.

\medskip


\end{abstract}

\section{Introduction}

Throughout this paper we assume that all graphs
considered are loopless, while parallel edges
are allowed. Let $G=(V, E)$ be a graph. For
any $e\in E(G)$, let $G-e$ be the graph obtained
from $G$ by deleting $e$. Let $G/e$ be the edge
contraction of $G$ by $e$, with loops deleted.
For any vertex set $S\subseteq V(G)$, let $G/S$ be
the graph obtained from $G$ by identifying all the
vertices in $S$, with resulting loops deleted.
Similarly, for any edge set $F\subseteq E(G)$,
let $G/F$ be the graph obtained from $G$ by
contracting all the edges in $F$, with resulting loops deleted.

Suppose $G$ is a weighted graph with edge weight $w_e$ for each edge $e$.
For any $F\subseteq E$,
define $w_F=\prod\limits_{e\in F} w_e$.
Let $\mathcal{T}(G)$ denote the set of spanning trees of $G$
and let
$\tau(G)=\sum\limits_{T\in \mathcal{T}(G)} w_T$.
It is easy
to see that for an unweighted graph $G$
(weighted graph with unit weight on
each edge), $\tau(G)$ is the number of spanning trees.
The first famous result for $\tau(G)$ is the Cayley's formula.

\begin{theorem}[Cayley's formula, \cite{Cayley}]\label{Cayley}
$$\tau(K_n)=n^{n-2}.$$
\end{theorem}

For a subgraph $H$ of $G$,
let $\mathcal{T}_{H}(G)$ denote
the set of spanning trees of $G$ containing
all edges in $H$ and let
$\tau_{H}(G)=\sum\limits_{T\in \mathcal{T}_{H}(G)} w_T$.
For unweighted graph $G$, $\tau_{H}(G)$ is
the number of spanning trees of $G$ containing $H$.
The following beautiful formula due to Moon generalizes Cayley's
formula by obtaining the number of spanning trees of $K_n$
containing a given spanning forest.

\begin{theorem}[\cite{Moon}, also see Problem 4.4 in \cite{Lovasz}]\label{forest}
For any spanning forest $F$ of $K_n$, if $c$ is the number of components
of $F$ and $n_1, n_2, \ldots, n_c$ are the orders of those components, then
$$\tau_{F}(K_n)=n^{c-2}\prod_{i=1}^c n_i.$$
\end{theorem}

Note that Cayley's formula is the special case of Theorem \ref{forest}
that $F$ is an empty graph.

It is also well known that
$\tau(K_{m, n})=m^{n-1}n^{m-1}$
for any complete bipartite graph $K_{m, n}$
by Fiedler and Sedl\'{a}\v{c}ek \cite{FS}.
In this paper,
we prove the following two theorems as bipartite
analogues of Moon's formula (Theorem \ref{forest}).

\begin{theorem}\label{matching}
For any matching $M$ of size $k$ in $K_{m, n}$,
$$\tau_M(K_{m, n})=(m+n)^{k-1}(m+n-k)m^{n-k-1}n^{m-k-1}.$$
\end{theorem}

\begin{theorem}\label{tree}
Let $T$ be any tree which is a subgraph of $K_{m, n}$.
Then
$$
\tau_T(K_{m, n})=(sn+tm-st)m^{n-t-1}n^{m-s-1},
$$
where $s=|V(T)\cap X|$, $t=|V(T)\cap Y|$,
and $(X, Y)$ is the bipartition
of $K_{m, n}$ with $|X|=m$ and $|Y|=n$.
\end{theorem}

Effective resistance is a concept used in electric circuit
analysis to define the equivalent resistance between
two points in an electric network. In graph theory
it is also known as the resistance distance \cite{KR}
between two vertices of a weighted connected
graph. Here the graph considered can be viewed as an
electric network, where edge $uv$ with edge weight
$w_{uv}>0$ corresponds to a resistor with resistance
$r_{uv}=\frac{1}{w_{uv}}$, with a voltage source
connected to a pair of vertices.

Recently, Ye and Yan \cite{YY} studied the resistance
distance between any two vertices of
the graph $K_{n,n}-pK_2$
$(p\leq n)$. In this paper, we extend their
results to any nearly complete bipartite graph
$K_{m,n}-pK_2$ $(p\leq \min\{m,n\})$,
that is,
the graph obtained from the complete bipartite graph $K_{m, n}$
by deleting a matching of size $p$.

As a corollary, the number of spanning trees of any nearly complete bipartite graph
is obtained as follows.

\begin{theorem}\label{nearly}
For any integer $p\le \min\{m,n\}$,
$$
\tau(K_{m,n}-pK_2)=(mn-m-n+p)(mn-m-n)^{p-1}m^{n-p-1}n^{m-p-1}.
$$
\end{theorem}

Theorem \ref{nearly} can be viewed as a ``dual'' result of Theorem \ref{matching},
since $\tau(K_{m,n}-pK_2)$
equals the number of those spanning trees of $K_{m,n}$ contain
no edge of a fixed matching of size $p$.

Shi and Chen~\cite{SC}
studied the problem of computing resistance
distances and Kirchhoff index (which will be introduced in the next section)
of graphs with an involution
and obtained the Kirchhoff index of $G(n,n,p)=K_{n,n}-pK_2$.
We extend their result to all nearly complete bipartite graphs $G(m,n,p)=K_{m,n}-pK_2$.

\section{Electrical network}

The main tool used in proving our results is the theory of
electrical network. In this section, we introduce
some basics of the theory of electrical network
and a nice relation between electrical network and
spanning trees of graphs.

For readers who are not familiar with basics of electricity,
we refer to \cite{Gervacio}. Doyle and Snell's book \cite{DS}
provides rich materials on random walks and electric networks.
Wagner's excellent notes \cite{Wagner} also contain further materials.
Vos's master thesis \cite{Vos} is a well written survey on how to determining the
effective resistance.

For any graph $G$, we view each edge as a resistor with
some assigned resistance. If a voltage source
is connected to two vertices of $G$, then each vertex
receives an electric potential, each edge receives
a current flow (possibly 0 for some edges), and the graph
becomes an electrical network. Followings
are some notations:
\begin{enumerate}
\renewcommand{\theenumi}{\rm (\roman{enumi})}

\item the electric potential at point $a$ is denoted by $v_a$;

\item the voltage, electric potential difference, or electric pressure from $a$ to $b$ is denoted by $U_{ab}=v_a-v_b$;

\item the electric current through a conductor represented by the edge $ab$ (from $a$ to $b$)
is denoted by $I_{ab}$;

\item the electrical resistance of a conductor represented by the edge $ab$ is denoted by $r_{ab}$.
\end{enumerate}

Then we know $U_{ab}=-U_{ba}$, $I_{ab}=-I_{ba}$, and $r_{ab}=r_{ba}$.

Now we introduce some basic laws in the theory of
electrical network.

\begin{theorem}[Ohm's law]\label{Ohm}
For the conductor between two points $a$ and $b$,
$$U_{ab}=I_{ab}r_{ab}.$$
\end{theorem}

\begin{theorem}[Kirchhoff's current law]\label{current}
The algebraic sum of currents in a network of conductors meeting at a point is zero.
\end{theorem}

\begin{theorem}[Kirchhoff's voltage law]\label{voltage}
The directed sum of the potential differences (voltages) around any closed circle is zero.
\end{theorem}

Ohm's law easily yields formulae for the equivalent resistance of resistors in series or in parallel.

\begin{figure}[htbp]
\centering
\includegraphics[width=8cm]{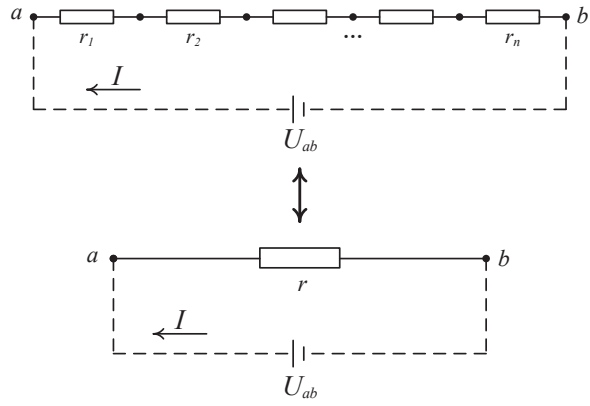}
\renewcommand{\figurename}{Fig.}
\caption{{\footnotesize Resistors in series.
}}\label{series}
\end{figure}

Figure \ref{series} shows $n$ resistors in series.
If we view the resistors in series as one single resistor with resistance $r$,
then $U_{ab}=Ir=I(r_1+r_2+\cdots+r_n)$,
implying that $r=r_1+r_2+\cdots+r_n$, which is called the series law.

\begin{figure}[htbp]
\centering
\includegraphics[width=12cm]{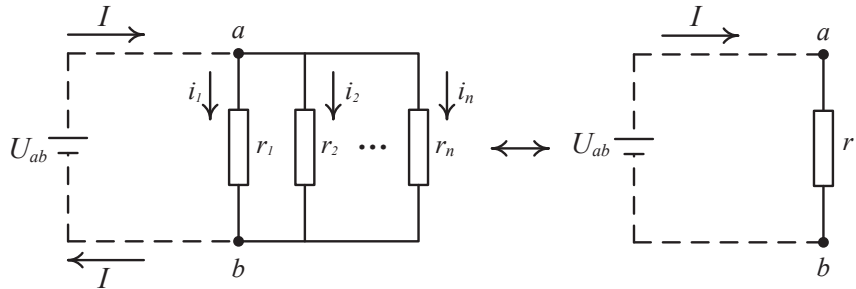}
\renewcommand{\figurename}{Fig.}
\caption{{\footnotesize Resistors in parallel.
}}\label{parallel}
\end{figure}

Figure \ref{parallel} shows $n$ resistors in parallel.
If we view the resistors in parallel as one single resistor with resistance $r$,
then $U_{ab}=Ir=i_1r_1=i_2r_2=\cdots=i_nr_n$,
implying that $I=\frac{U_{ab}}{r}$ and $i_i=\frac{U_{ab}}{r_i}$ for all $i=1,2,\cdots,n$.
Since $I=i_1+i_2+\cdots+i_n$, we obtain
$\frac{U_{ab}}{r}=\frac{U_{ab}}{r_1}+\frac{U_{ab}}{r_2}+\cdots+\frac{U_{ab}}{r_n}$, that is,
$\frac{1}{r}=\frac{1}{r_1}+\frac{1}{r_2}+\cdots+\frac{1}{r_n}$, which is called the parallel law.

The basic idea of effective resistance is established
by letting a voltage source connects vertices
$u$ and $v$.
Let $I$ be the electric current from the voltage source to $u$
(Kirchhoff's current law guarantees that it equals the
electric current from $v$ to the voltage source).
The voltage source effectively sees the network between
$u$ and $v$ as a single resistor connecting $u$ and $v$,
through which the electric current flows. The effective resistance
is basically the resistance of that single resistor.

\begin{definition}
Let $G$ be a graph with assigned resistances for edges. Connect a voltage source (a battery) between vertices $u$ and $v$.
Let $I$ be the electric current from the voltage source to $u$.
The effective resistance between vertices $u$ and $v$,
$R_{uv}(G)$ (or $R_{uv}$ if there is no confusion), is defined by the following:
$$R_{uv}(G)=\frac{U_{uv}}{I}.$$
\end{definition}

\begin{remark}
It is easy to see that $R_{uv}$ is independent of the value of the voltage source we added.
\end{remark}

In graph theory, sometimes we call $R_{uv}$
the resistance distance \cite{KR} between $u$ and $v$.
This is because the effective resistance function $R$ is proved
to be a metric on graphs \cite{Sharpe}:

$R: V(G)\times V(G) \rightarrow [0, +\infty)$
is a function that satisfies
\begin{enumerate}
\renewcommand{\theenumi}{\rm (\roman{enumi})}

\item non-negativity: $R_{xy}\geq 0$ for all $x, y\in V(G)$,

\item identity of indiscernibles: $R_{xy}=0$ if and only if $x=y$,

\item symmetry: $R_{xy}=R_{yx}$ for all $x, y\in V(G)$,

\item triangle inequality: $R_{xy}\leq R_{xz}+R_{yz}$ for all $x, y, z\in V(G).$
\end{enumerate}

In \cite{KR}, after named $R_{uv}$ the resistance distance, Klein and Randi\'{c}
defined the Kirchhoff index $Kf(G)$ of $G$ as the sum of resistance distances
between each pair of vertices of $G$, that is,
$$Kf(G)=\sum_{\{u, v\}: u, v\in V(G)} R_{uv}.$$

See \cite{Gervacio, SC, QZD, Yang1, YY, Yang2} for some recent researches on
effective resistance and Kirchhoff index.

Now we introduce a nice relation between electrical network and
spanning trees of graphs.

\begin{theorem}[\cite{SR}]\label{res-spanning}
Let $G=(V, E, \{w_e\})$ be a simple weighted graph with weight
$w_{ij}=\frac{1}{r_{ij}}$ if $ij\in E$, and $w_{ij}=0$
if $ij\notin E$, where $r_{ij}$ is the resistance of edge $ij\in E$.
For any two vertices $u, v\in V$,
the effective resistance between $u$ and $v$ is
$$R_{uv}=\frac{\tau(G/\{u,v\})}{\tau(G)}.$$
\end{theorem}

\section{Spanning trees in $K_{m, n}$ containing a certain matching}

In this section, we aim to prove Theorem \ref{matching}.

For any graph $G$ and any $F\subseteq E(G)$,
it is easy to see that $\tau_F(G)=\tau(G/F)$.
So for any two matchings $M_1$ and $M_2$ in $K_{m, n}$ with
the same size, $\tau_{M_1}(K_{m, n})=\tau_{M_2}(K_{m, n})$
because $K_{m, n}/M_1$ and $K_{m, n}/M_2$ are isomorphic.
So $\tau_M(K_{m, n})$ only depend on the size of $M$.
Now let $\tau_k(K_{m, n})$ be the number of spanning trees of
$K_{m, n}$ containing any given matching of size $k$.
Then $\tau_0(K_{m, n})=\tau(K_{m, n})=m^{n-1}n^{m-1}$.

\begin{lemma}\label{t1}
$\tau_1(K_{m, n})=\frac{m+n-1}{mn}\cdot\tau(K_{m, n}).$
\end{lemma}

\begin{proof}
For any $e\in E(K_{m,n})$, we prove there are
$\frac{m+n-1}{mn}\tau(K_{m, n})$ spanning trees
contain $e$. Note that $K_{m,n}$ has $mn$ edges,
and $\tau(K_{m, n})$ spanning trees.
Construct a bipartite graph $H$ with bipartition
$(X, Y)$, where $X$ is the set of edges in
$K_{m,n}$ and $Y$ is the set of
spanning trees in $K_{m,n}$,
and with edge set $\{\{e,T\}: e\in X, T\in Y, e\in E(T)\}$.
By symmetry, $H$ is a biregular graph. Each vertex
in $X$ has degree $\tau_1(K_{m, n})$ and each vertex
in $Y$ has degree $m+n-1$. Therefore, we have
$$|X|\cdot \tau_1(K_{m, n})=|Y|\cdot(m+n-1).$$
By definition, $|X|=mn$ and $|Y|=\tau(K_{m,n})$.
Thus,
$$mn\cdot \tau_1(K_{m, n})=(m+n-1)\cdot \tau(K_{m, n}),$$
that is,
$$\tau_1(K_{m, n})=\frac{m+n-1}{mn}\cdot \tau(K_{m, n}).$$
\end{proof}

\begin{theorem}\label{mathcing-ratio}
For any $1\leq k\leq \min\{m, n\}-1$,
$$\frac{\tau_{k+1}(K_{m, n})}{\tau_k(K_{m, n})}=\frac{(m+n)(m+n-k-1)}{mn(m+n-k)}.$$
\end{theorem}

\begin{proof}
For any matching $M$ with size $k$,
we consider the graph $K_{m, n}/M$.
Assume that $(X,Y)$ is the bipartition of $K_{m,n}$,
where $X=\{s, x_1, x_2, \ldots, x_{m-1}\}$
and $Y=\{t, y_1, y_2, \ldots, y_{n-1}\}$,
and that $M=\{x_{m-i}y_{n-i}:i=1,2,\ldots,k\}$.
Then the vertex set of $K_{m,n}/M$
is
$\{s, x_1, x_2, \ldots, x_{m-k-1}\}
\cup \{t, y_1, y_2, \ldots, y_{n-k-1}\}
\cup \{z_1, z_2, \ldots, z_k\}$,
where each $z_i$ is the vertex
obtained after contracting edge $x_{m-i}y_{n-i}$ for $i=1,2,\ldots,k\}$.
An example of $K_{m, n}/M$ is shown in
Figure \ref{matching-contract}.

\begin{figure}[htbp]
\centering
\includegraphics[width=10cm]{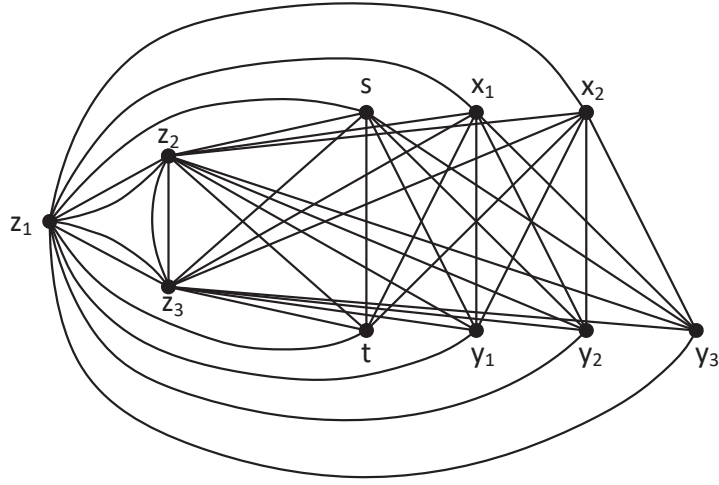}
\renewcommand{\figurename}{Fig.}
\caption{{\footnotesize $K_{6, 7}/M$,
where $M$ is a matching of size $3$.
}}\label{matching-contract}
\end{figure}

We view $K_{m, n}/M$ as an electrical network. Assume a unit
current enters at $s$ and leaves at $t$
and each edge has unit
resistance.

Applying the parallel law, we replace the two parallel edges between $z_i$ and $z_j$
($i\neq j$) by one single edge $z_iz_j$ with resistance $\frac{1}{2}$,
and denote the resulting graph by $G'$. Then $R_{st}(K_{m, n}/M)=R_{st}(G')$.
In what follows, we consider the electrical network $G'$ with a unit
current enters at $s$ and leaves at $t$.

Write $N=m+n-k$.
We assign a current to each edge as follows:
\begin{displaymath}
\left\{ \begin{array}{ll}
I_{st}=w_0=\frac{1}{m}+\frac{1}{n}-\frac{1}{mN}-\frac{1}{nN}, &  \\
I_{sy_i}=w_1=\frac{1}{n}-\frac{1}{mN}-\frac{1}{nN}, & \textrm{$1\leq i\leq n-k-1$,} \\
I_{sz_i}=w_2=\frac{1}{n}-\frac{1}{nN}, & \textrm{$1\leq i\leq k$,} \\
I_{tx_i}=w_3=\frac{1}{mN}+\frac{1}{nN}-\frac{1}{m}, & \textrm{$1\leq i\leq m-k-1$,} \\
I_{tz_i}=w_4=\frac{1}{mN}-\frac{1}{m}, & \textrm{$1\leq i\leq k$,} \\
I_{x_iy_j}=w_5=-\frac{1}{mN}-\frac{1}{nN}, & \textrm{$1\leq i\leq m-k-1$, $1\leq j\leq n-k-1$,} \\
I_{z_ix_j}=w_6=\frac{1}{nN}, & \textrm{$1\leq i\leq k$, $1\leq j\leq m-k-1$,} \\
I_{z_iy_j}=w_7=-\frac{1}{mN}, & \textrm{$1\leq i\leq k$, $1\leq j\leq n-k-1$,} \\
I_{z_iz_j}=0, & \textrm{$1\leq i\neq j\leq k$.}
\end{array} \right.
\end{displaymath}

First we check Kirchhoff's current law for each vertex:
\begin{equation*}\begin{split}
& \mbox{at vertex }s:
-1+\sum_{u: u\sim s} I_{su}=-1+w_0+(n-k-1)w_1+kw_2=0; \\
&\mbox{at vertex } t: 1+\sum_{u: u\sim t} I_{tu}=1-w_0+(m-k-1)w_3+kw_4=0; \\
&\mbox{at vertex } x_i ~(1\leq i\leq m-k-1): \sum_{u: u\sim x_{i}} I_{x_{i}u}=(n-k-1)w_5-w_3-kw_6=0; \\
&\mbox{at vertex } y_i ~(1\leq i\leq n-k-1): \sum_{u: u\sim y_{i}} I_{y_{i}u}=w_1+(m-k-1)w_5+kw_7=0; \\
&\mbox{at vertex } z_i ~(1\leq i\leq k): \sum_{u: u\sim z_{i}} I_{z_{i}u}=(m-k-1)w_6+(n-k-1)w_7-w_2-w_4=0.
\end{split}\end{equation*}

Then we check Kirchhoff's voltage law for cycles.

A cycle basis of a graph is a minimal set of cycles
such that each Eulerian subgraph can be expressed
as a symmetric difference of cycles in this set.

It is easy to see that if cycle $C$ is the
symmetric difference of $k$ cycles, and the
Kirchhoff's voltage law holds for all these
$k$ cycles, then the Kirchhoff's voltage law
also holds for $C$. That is to say, we only
need to check those cycles in a cycle basis.

Let $T$ be a spanning tree of $K_{mn}/M$
that $T$ is the star with center vertex $z_1$.
For any $e\in E(K_{mn}/M)-E(T)$, we have a unique
cycle $C_{e}$ consisting of $e$ together with
the unique path in $T$ connecting the endpoints of $e$.
All these cycles are called fundamental cycles
and they form a special basis for the cycle space
called fundamental cycle basis. We check Kirchhoff's
voltage law for all these fundamental cycles.

$z_1\rightarrow s\rightarrow t\rightarrow z_1$
(the fundamental cycle corresponding to the edge $st$):
\begin{eqnarray*}
\ U_{z_1s}+U_{st}+U_{tz_1} & = & I_{z_1s}r_{z_1s}+I_{st}r_{st}+I_{tz_1}r_{tz_1} \\
& = & -w_2\cdot 1+w_0\cdot 1+w_4\cdot 1 \\
& = & -\frac{1}{n}+\frac{1}{nN}+\frac{1}{m}+\frac{1}{n}-\frac{1}{mN}-\frac{1}{nN}+\frac{1}{mN}-\frac{1}{m}\\
& = & 0.
\end{eqnarray*}

$z_1\rightarrow s\rightarrow y_i\rightarrow z_1$
(the fundamental cycle corresponding to edge $sy_i$):
\begin{eqnarray*}
\ U_{z_1s}+U_{sy_i}+U_{y_iz_1}
& = & I_{z_1s}r_{z_1s}+I_{sy_i}r_{sy_i}+I_{y_iz_1}r_{y_iz_1} \\
& = & -w_2\cdot 1+w_1\cdot 1-w_7\cdot 1 \\
& = & -\frac{1}{n}+\frac{1}{nN}+\frac{1}{n}-\frac{1}{mN}-\frac{1}{nN}+\frac{1}{mN} \\
& = & 0.
\end{eqnarray*}

$z_1\rightarrow s\rightarrow z_i\rightarrow z_1$
(the fundamental cycle corresponding to the edge $sz_i$):
\begin{eqnarray*}
\ U_{z_1s}+U_{sz_i}+U_{z_iz_1}
& = & I_{z_1s}r_{z_1s}+I_{sz_i}r_{sz_i}+I_{z_iz_1}r_{z_iz_1} \\
& = & -w_2\cdot 1+w_2\cdot 1-0\cdot 1 \\
& = & 0.
\end{eqnarray*}

$z_1\rightarrow t\rightarrow x_i\rightarrow z_1$
(the fundamental cycle corresponding to the edge $tx_i$):
\begin{eqnarray*}
\ U_{z_1t}+U_{tx_i}+U_{x_iz_1}
& = & I_{z_1t}r_{z_1t}+I_{tx_i}r_{tx_i}+I_{x_iz_1}r_{x_iz_1} \\
& = & -w_4\cdot 1+w_3\cdot 1-w_6\cdot 1 \\
& = & -\frac{1}{mN}+\frac{1}{m}+\frac{1}{mN}+\frac{1}{nN}-\frac{1}{m}-\frac{1}{nN} \\
& = & 0.
\end{eqnarray*}

$z_1\rightarrow t\rightarrow z_i\rightarrow z_1$
(the fundamental cycle corresponding to the edge $tz_i$):
\begin{eqnarray*}
\ U_{z_1t}+U_{tz_i}+U_{z_iz_1}
& = & I_{z_1t}r_{z_1t}+I_{tz_i}r_{tz_i}+I_{z_iz_1}r_{z_iz_1} \\
& = & -w_4\cdot 1+w_4\cdot 1-0\cdot 1 \\
& = & 0.
\end{eqnarray*}

$z_1\rightarrow x_i\rightarrow y_j\rightarrow z_1$
(the fundamental cycle corresponding to the edge $x_iy_j$):
\begin{eqnarray*}
\ U_{z_1x_i}+U_{x_iy_j}+U_{y_jz_1}
& = & I_{z_1x_i}r_{z_1x_i}+I_{x_iy_j}r_{x_iy_j}+I_{y_jz_1}r_{y_jz_1} \\
& = & w_6\cdot 1+w_5\cdot 1-w_7\cdot 1 \\
& = & \frac{1}{nN}-\frac{1}{mN}-\frac{1}{nN}+\frac{1}{mN} \\
& = & 0.
\end{eqnarray*}

$z_1\rightarrow x_i\rightarrow z_j\rightarrow z_1$
(the fundamental cycle corresponding to the edge $x_iz_j$):
\begin{eqnarray*}
\ U_{z_1x_i}+U_{x_iz_j}+U_{z_jz_1}
& = & I_{z_1x_i}r_{z_1x_i}+I_{x_iz_j}r_{x_iz_j}+I_{z_jz_1}r_{z_jz_1} \\
& = & w_6\cdot 1-w_6\cdot 1-0\cdot 1 \\
& = & 0.
\end{eqnarray*}

$z_1\rightarrow y_i\rightarrow z_j\rightarrow z_1$
(the fundamental cycle corresponding to the edge $y_iy_j$):
\begin{eqnarray*}
\ U_{z_1y_i}+U_{y_iz_j}+U_{z_jz_1}
& = & I_{z_1y_i}r_{z_1y_i}+I_{y_iz_j}r_{y_iz_j}+I_{z_jz_1}r_{z_jz_1} \\
& = & w_7\cdot 1-w_7\cdot 1-0\cdot 1 \\
& = & 0.
\end{eqnarray*}

$z_1\rightarrow z_i\rightarrow z_j\rightarrow z_1, i\neq j$
(the fundamental cycle corresponding to the edge $z_iz_j$):
\begin{eqnarray*}
\ U_{z_1z_i}+U_{z_iz_j}+U_{z_jz_1} & = & I_{z_1z_i}r_{z_1z_i}+I_{z_iz_j}r_{z_iz_j}+I_{z_jz_1}r_{z_jz_1} \\
& = & 0\cdot \frac{1}{2}+0\cdot \frac{1}{2}+0\cdot \frac{1}{2} \\
& = & 0.
\end{eqnarray*}

Therefore, the resistances and currents we assigned satisfy the Kirchhoff's laws.
By the definition of effective resistance, we obtain
$$R_{st}(K_{m, n}/M)=R_{st}(G')=\frac{U_{st}}{1}=I_{st}r_{st}=w_0
=\frac{(m+n)(m+n-k-1)}{mn(m+n-k)}.$$
Then, following from Theorem \ref{res-spanning}, we have
\begin{eqnarray*}
\ \frac{\tau_{k+1}(K_{m, n})}{\tau_{k}(K_{m, n})} & = &
\frac{\tau\left((K_{m, n}/M)/st\right)}{\tau(K_{m, n}/M)} \\
& = & \frac{\tau\left((K_{m, n}/M)/\{s, t\}\right)}{\tau(K_{m, n}/M)} \\
& = & R_{st}(K_{m, n}/M) \\
& = & \frac{(m+n)(m+n-k-1)}{mn(m+n-k)}.
\end{eqnarray*}
\end{proof}

\noindent
{\bf Proof of Theorem \ref{matching}.}
Since $\tau (K_{m, n})=m^{n-1}n^{m-1}$, Theorem \ref{matching} holds when $k=0$.

Lemma \ref{t1} shows that Theorem \ref{matching} holds when $k=1$.

When $2\leq k\leq \min\{m, n\}$, by using Theorem \ref{mathcing-ratio}, we obtain
\begin{eqnarray*}
\ \tau_M(K_{m, n}) & = & \tau_k(K_{m, n}) \\
& = & \tau_1(K_{m, n})\cdot \prod_{i=1}^{k-1} \frac{\tau_{i+1}(K_{m, n})}{\tau_i(K_{m, n})} \\
& = & (m+n-1)m^{n-2}n^{m-2}\cdot \prod_{i=1}^{k-1} \frac{(m+n)(m+n-i-1)}{mn(m+n-i)}\\
& = & (m+n)^{k-1}(m+n-k)m^{n-k-1}n^{m-k-1}.
\end{eqnarray*}
The proof is complete. {\hfill$\Box$}

\section{Spanning trees in $K_{m, n}$ containing a certain tree}

In this section, we aim to prove Theorem \ref{tree}.

Let $(X, Y)$ be the bipartition of $K_{m, n}$ with $|X|=m$ and $|Y|=n$.
Note that for any $F\subseteq E(G)$, there exists $\tau_F(G)=\tau(G/F)$.
We observe that for any two trees $T_1$ and $T_2$ in $K_{m, n}$,
if $|V(T_1)\cap X|=|V(T_2)\cap X|$, and $|V(T_1)\cap Y|=|V(T_2)\cap Y|$,
then $\tau_{T_1}(K_{m, n})=\tau_{T_2}(K_{m, n})$
since $K_{m, n}/T_1$ and $K_{m, n}/T_2$ are isomorphic,
where for a subgraph $H$ in a graph $G$,
$G/H$ is the graph $G/E(H)$.
So when we count the number of spanning trees in
$K_{m, n}$ containing a certain tree $T$, the only
thing matters is the pair of numbers $s$ and $t$,
where $s=|V(T)\cap X|$ and $t=|V(T)\cap Y|$.

Now let $\tau_{s,t}(K_{m, n})$ be the number of spanning trees of
$K_{m, n}$ containing any given $T$ as a subgraph such that
$|V(T)\cap X|=s$ and $|V(T)\cap Y|=t$.
Then $\tau_{0, 0}(K_{m, n})=
\tau_{1, 0}(K_{m, n})=\tau_{0, 1}(K_{m, n})=\tau(K_{m, n})$.
Lemma \ref{t1} tells that $\tau_{1, 1}(K_{m, n})=\frac{m+n-1}{mn}\cdot\tau(K_{m, n}).$

\begin{theorem}\label{tree-ratio}
(1) For any $1\leq s\leq m$, $1\leq t\leq n-1$,
$$\frac{\tau_{s, t+1}(K_{m, n})}{\tau_{s, t}(K_{m, n})}=\frac{sn+(m-s)(t+1)}{m[sn+(m-s)t]}.$$

(2) For any $1\leq s\leq m-1$, $1\leq t\leq n$,
$$\frac{\tau_{s+1, t}(K_{m, n})}{\tau_{s, t}(K_{m, n})}=\frac{tm+(n-t)(s+1)}{n[tm+(n-t)s]}.$$
\end{theorem}

\begin{proof}
We only prove part (1), after that, part (2) remains the same.

Let tree $T$ be a subgraph of $K_{m, n}$.
Let $s=|V(T)\cap X|$, and $t=|V(T)\cap Y|$,
where $(X, Y)$ is the bipartition
of $K_{m, n}$ with $|X|=m$ and $|Y|=n$.
Assume that $X=\{x_1, x_2, \ldots, x_m\}$,
$Y=\{y_0, y_1, \ldots, y_{n-1}\}$, and
$V(T)=\{x_{m-s+1}, x_{m-s+2}, \ldots, x_m\}
\cup \{ y_{n-t}, y_{n-t+1}, \ldots, y_{n-1}\}$.
Now we consider the graph $K_{m, n}/T$, in which
all vertices in $V(T)$ are identified
as a single vertex $z$, as shown in
Figure \ref{tree-contract}.

\begin{figure}[htbp]
\centering
\includegraphics[width=10cm]{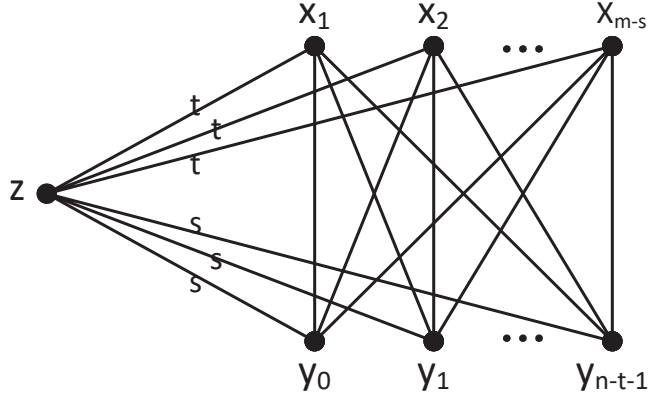}
\renewcommand{\figurename}{Fig.}
\caption{{\footnotesize
Graph $K_{m, n}/T$, where $|V(T)\cap X|=s$ and $|V(T)\cap Y|=t$. Edges
labelled by $s$ or $t$ means they are multiple edges with multiplicity $s$ or $t$.
}}\label{tree-contract}
\end{figure}

We view the $K_{m, n}/T$ as an electrical network.
Assume a unit
current enters at $z$ and leaves at $y_0$
and each edge of $K_{m, n}/T$ has
unit resistance.
Our aim is to obtain the effective resistance $R_{zy_0}$.

Applying the parallel law,
we can replace $q$ parallel edges between two vertices
$u$ and $v$ by a single edge $uv$ with resistance $\frac{1}{q}$.
Thus, for any $i$ with $1\leq i \leq m-s$,
all $t$ parallel edges between $z$ and $x_i$
can be replaced by one single edge $zx_i$ with resistance $\frac{1}{t}$,
and for any $j$ with $0\leq j \leq n-t-1$,
all $s$ parallel edges between $z$ and $y_j$
can be replaced by one single edge $zy_j$ with resistance $\frac{1}{s}$.
We denote the resulting graph by $G''$. Then $R_{zy_0}(K_{m, n}/T)=R_{zy_0}(G'')$.
In what follows, we consider the electrical network $G''$ with a unit
current enters at $z$ and leaves at $y_0$.

Write $N=sn+tm-st$. To each edge of $G''$
we assign a current as follows:

\begin{displaymath}
\left\{ \begin{array}{ll}
I_{zy_0}=w_0=\frac{s}{m}+\frac{s(m-s)}{mN}, &  \\
I_{zy_i}=w_1=\frac{s(m-s)}{mN}, & \textrm{$1\leq i\leq n-t-1$,} \\
I_{zx_i}=w_2=\frac{t}{N}, & \textrm{$1\leq i\leq m-s$,} \\
I_{x_iy_0}=w_3=\frac{1}{m}-\frac{s}{mN}, & \textrm{$1\leq i\leq m-s$,} \\
I_{x_iy_j}=w_4=-\frac{s}{mN}, & \textrm{$1\leq i\leq m-s$, $1\leq j\leq n-t-1$.} \\
\end{array} \right.
\end{displaymath}

First we check Kirchhoff's current law for each vertex:
\begin{equation*}
\begin{array}{ll}
\mbox{at vertex }z:
&-1+\sum\limits_{u: u\sim z} I_{zu}=-1+(m-s)w_2+w_0+(n-t-1)w_1=0; \\
\mbox{at vertex } y_0:
&1+\sum\limits_{u: u\sim y_0} I_{y_0u}=1-w_0-(m-s)w_3=0; \\
\mbox{at vertex } x_i ~(1\leq i\leq m-s):
&\sum\limits_{u: u\sim x_{i}} I_{x_{i}u}=-w_2+w_3+(n-t-1)w_4=0; \\
\mbox{at vertex } y_i ~(1\leq i\leq n-t-1):
&\sum\limits_{u: u\sim y_{i}} I_{y_{i}u}=-w_1-(m-s)w_4=0.
\end{array}
\end{equation*}

Then we check Kirchhoff's voltage law for cycles.
Similar to the proof of Theorem \ref{mathcing-ratio},
we only need to check those cycles in a cycle basis.

Let $T$ be a spanning tree of $G''$
that is the star with $z_1$ as its center.
For any $e\in E(G'')-E(T)$, we have a unique
cycle $C_{e}$ consisting of $e$ and
the unique path in $T$ connecting the endpoints of $e$.
All these cycles form a
a special basis for the cycle space.
We check Kirchhoff's
voltage law for all these cycles:

for the cycle $C_{zx_iy_0}$ ($z\rightarrow x_i\rightarrow y_0\rightarrow z$):

\begin{eqnarray*}
\ U_{zx_i}+U_{x_iy_0}+U_{y_0z} & = & I_{zx_i}r_{zx_i}+I_{x_iy_0}r_{x_iy_0}+I_{y_0z}r_{y_0z} \\
& = & w_2\cdot \frac{1}{t}+w_3\cdot 1-w_0\cdot \frac{1}{s} \\
& = & \frac{1}{N}+\frac{1}{m}-\frac{s}{mN}-\left(\frac{1}{m}+\frac{m-s}{mN}\right) \\
& = & 0;
\end{eqnarray*}

and for the cycle $C_{zx_iy_j}$
($z\rightarrow x_i\rightarrow y_j\rightarrow z$),
where $j\neq 0$:

\begin{eqnarray*}
\ U_{zx_i}+U_{x_iy_j}+U_{y_jz} & = & I_{zx_i}r_{zx_i}+I_{x_iy_j}r_{x_iy_j}+I_{y_jz}r_{y_jz} \\
& = & w_2\cdot \frac{1}{t}+w_4\cdot 1-w_1\cdot \frac{1}{s} \\
& = & \frac{1}{N}-\frac{s}{mN}-\frac{s(m-s)}{mN}\cdot\frac{1}{s} \\
& = & 0.
\end{eqnarray*}

Therefore, the resistances and currents we assigned satisfy the Kirchhoff's laws.
By the definition of effective resistance, we obtain
$$R_{zy_0}(K_{m, n}/T)=R_{zy_0}(G'')=\frac{U_{zy_0}}{1}=I_{zy_0}r_{zy_0}=w_0\cdot\frac{1}{s}
=\frac{sn+(m-s)(t+1)}{m[sn+(m-s)t]}.$$
Then, following from Theorem \ref{res-spanning}, we have
\begin{eqnarray*}
\ \frac{\tau_{s, t+1}(K_{m, n})}{\tau_{s, t}(K_{m, n})}
& = & \frac{\tau\left((K_{m, n}/T)/zy_0\right)}{\tau(K_{m, n}/T)} \\
& = & \frac{\tau\left((K_{m, n}/T)/\{z, y_0\}\right)}{\tau(K_{m, n}/T)} \\
& = & R_{zy_0}(K_{m, n}/T) \\
& = & \frac{sn+(m-s)(t+1)}{m[sn+(m-s)t]}.
\end{eqnarray*}
\end{proof}

\noindent
{\bf Proof of Theorem \ref{tree}.}
When $1\leq s\leq m$, and $1\leq t\leq n$,
by using Theorem \ref{tree-ratio}, we obtain
\begin{eqnarray}
\ \tau_{s, t}(K_{m, n}) & = & \tau_{s, 1}(K_{m, n})\cdot
\prod_{i=1}^{t-1} \frac{\tau_{s, i+1}(K_{m, n})}{\tau_{s, i}(K_{m, n})}
\nonumber  \\
& = & \tau_{1, 1}(K_{m, n})\cdot
\prod_{i=1}^{s-1} \frac{\tau_{i+1, 1}(K_{m, n})}{\tau_{i, 1}(K_{m, n})}\cdot
\prod_{i=1}^{t-1} \frac{\tau_{s, i+1}(K_{m, n})}{\tau_{s, i}(K_{m, n})}
\nonumber  \\
& = & \frac{m+n-1}{mn}\cdot\tau(K_{m, n})\cdot
\prod_{i=1}^{s-1} \frac{m+(n-1)(i+1)}{n[m+(n-1)i]}\cdot
\prod_{i=1}^{t-1} \frac{sn+(m-s)(i+1)}{m[sn+(m-s)i]}
\nonumber  \\
& = & \frac{m+n-1}{mn}\cdot\tau(K_{m, n})\cdot\frac{m+(n-1)s}{n^{s-1}(m+n-1)}\cdot
\frac{sn+(m-s)t}{m^{t-1}(sn+m-s)}
\nonumber \\ & = & \frac{sn+tm-st}{m^sn^t}\cdot\tau(K_{m, n}).
\label{th1-4-e1}
\end{eqnarray}

For $s=m$, $t=n$,
(\ref{th1-4-e1}) implies that
$\tau_{m, n}(K_{m, n})=\frac{\tau(K_{m, n})}{m^{n-1}n^{m-1}}$.
Since $\tau_{m, n}(K_{m, n})$
is the number of spanning trees
of $K_{m, n}$ which contain a given spanning tree,
obviously $\tau_{m, n}(K_{m, n})=1$.
It follows immediately
that $\tau(K_{m, n})=m^{n-1}n^{m-1}$.
By (\ref{th1-4-e1}) again,
when $2\leq s\leq m$, and $2\leq t\leq n$, we have
$\tau_{s,t}(K_{m, n})=(sn+tm-st)m^{n-t-1}n^{m-s-1}$.
Theorem \ref{tree} holds in this case.

Since $\tau_{0, 0}(K_{m, n})=\tau_{1, 0}(K_{m, n})=\tau_{0, 1}(K_{m, n})=\tau(K_{m, n})=m^{n-1}n^{m-1}$,
it is easy to see that Theorem \ref{tree} holds for $(s, t)\in \{(0,0), (1,0), (0,1)\}$.
It is not possible for $s=0, t\geq 2$ or $s\geq 2, t=0$.
So Theorem \ref{tree} holds for all remaining cases. The proof is complete. {\hfill$\Box$}

\begin{remark}
It is surprising that throughout the proof of Theorem \ref{tree},
we do not need the exact value of $\tau(K_{m, n})$ as the initial data.
On the contrary, we can deduce the formula
$\tau(K_{m, n})=m^{n-1}n^{m-1}$ from
the proof.
\end{remark}

Theorem \ref{tree} can also be deduced from Kirchoff's Matrix-Tree Theorem.
Let $G$ be a multigraph without loops. Let $A(G)$ be its adjacency matrix
and $D(G)$ its diagonal matrix. We call $L(G)=D(G)-A(G)$ the Laplacian
of $G$. Let $L(G)_{i,j}$ denote the matrix obtained from $L(G)$ by deleting
the $i$-th row and $j$-th column. Sometimes we call this submatrix a
reduced Laplacian of $G$.

\begin{theorem}[Kirchoff's Matrix-Tree Theorem, \cite{Kirchhoff}]\label{Matrix-Tree}
$$\tau(G)=(-1)^{i+j}\det(L(G)_{ij}).$$
\end{theorem}

\begin{lemma}[Schur's determinant identity, \cite{Schur}]\label{Schur}
Suppose $M$ is a square matrix that can be decomposed into blocks as
$$M=\left[
\begin{array}{ll}
{A} & {B} \\
{C} & {D}
\end{array}\right],$$
where $A$ and $D$ are square, and $D$ is invertible. Then
$$\det(M)=\det(D)\cdot\det(A-BD^{-1}C).$$
\end{lemma}

\noindent
{\bf Second proof of Theorem \ref{tree}.}
Let $K_{m,n}/T$ be the graph show in Figure \ref{tree-contract}.
It is easy to see that
\begin{equation}
D(K_{m,n}/T)=\text{diag}\{tm+sn-2st, \underbrace{n, \ldots, n}_{m-s}, \underbrace{m, \ldots, m}_{n-t}\},
\end{equation}
and
$$A(K_{m,n}/T)=\left[
\begin{array}{lll}
{0}      & {{\bf 1}}                     & {{\bf 1}} \\
{\bf{1}} & {{\bf 0}_{m-s}}               & {{\bf 1}_{(m-s)\times(n-t)}} \\
{\bf{1}} & {{\bf 1}_{(n-t)\times(m-s)}}  & {{\bf 0}_{n-t}}
\end{array}\right].$$
Therefore,
$$L(K_{m,n}/T)_{1,1}=\left[
\begin{array}{ll}
{nI_{m-s}}                   & {-{\bf 1}_{(m-s)\times(n-t)}} \\
{-{\bf 1}_{(n-t)\times(m-s)}} & {mI_{n-t}}
\end{array}\right].$$

By Theorem~\ref{Matrix-Tree} and Lemma~\ref{Schur},
we obtain
\begin{eqnarray*}
\ \tau_{s,t}(K_{m, n})
& = & \tau(K_{m,n}/T) \\
& = & \det
\left(
\left[
\begin{array}{ll}
{nI_{m-s}}                   & {-\bf{1}_{(m-s)\times(n-t)}} \\
{-\bf{1}_{(n-t)\times(m-s)}} & {mI_{n-t}}
\end{array}\right]\right) \\
& = & \det(mI_{n-t})\cdot \det\left(nI_{m-s}-
(-{\bf 1}_{(m-s)\times(n-t)})(mI_{n-t})^{-1}(-{\bf 1}_{(n-t)\times(m-s)})\right) \\
& = & m^{n-t}\det\left(nI_{m-s}-
\frac{1}{m}(-{\bf 1}_{(m-s)\times(n-t)})\cdot(-{\bf 1}_{(n-t)\times(m-s)}) \right) \\
& = & m^{n-t}\det\left(nI_{m-s}-\frac{n-t}{m}{\bf 1}_{(m-s)} \right) \\
& = & m^{n-t}\cdot\left(n-\frac{n-t}{m}+(m-s-1)\cdot\left(-\frac{n-t}{m}\right)\right)\cdot n^{m-s-1} \\
& = & (sn+tm-st)m^{n-t-1}n^{m-s-1}.
\end{eqnarray*}
The second-to-last equality above follows from
the following identity:
$$
\det\left(\left[
\begin{array}{llll}
{a}      & {b}      & {\cdots} & {b} \\
{b}      & {a}      & {\cdots} & {b} \\
{\vdots} & {\vdots} & {\ddots} & {\vdots} \\
{b}      & {b}      & {\cdots} & {a} \\
\end{array}\right]_{n\times n}\right)
=\left(a+(n-1)b\right)(a-b)^{n-1}.
$$
{\hfill$\Box$}

\section{Effective resistances in nearly complete bipartite graphs}

In this section,
we first introduce some known results on
effective resistances, and then apply them
to obtain the effective resistances of nearly
complete bipartite graphs. As applications,
we obtain the number of spanning trees (Theorem \ref{nearly}) and the
Kirchhoff indices of nearly complete bipartite graphs.

\subsection{Known results on effective resistances
\label{sec5-1}}

Foster \cite{Foster} obtained a famous result
on the sum of effective resistances for all edges as follows.

\begin{theorem}[\cite{Foster}]\label{global}
Let $G$ be a connected graph with $n$ vertices
in which each edge $ij$ receives unit resistance.
Then
$$\sum_{ij\in E(G)} R_{ij}=n-1.$$
\end{theorem}

Tetali \cite{Tetali91, Tetali94} gave two beautiful probabilistic proofs of Foster's Theorem.
It is natural to see this fact by considering the relation between effective resistance
and the number of spanning trees of graphs. We give a short proof of Foster's Theorem
based on Theorem \ref{res-spanning}.

\noindent
{\bf Proof of Theorem \ref{global}.}
By Theorem \ref{res-spanning}, we have
$R_{ij}=\frac{\tau(G/\{i,j\})}{\tau(G)}
=\frac{\tau(G/ij)}{\tau(G)}$ for each $ij\in E(G)$.
Then
$$
\sum_{ij\in E(G)} R_{ij}
=\sum_{e\in E(G)} \frac{\tau(G/e)}{\tau(G)}=
\frac 1{\tau(G)} \sum_{e\in E(G)}\tau(G/e).
$$
Since for any $e\in E(G)$,
$\tau(G/e)$ is  the number of spanning trees of $G$ containing edge $e$,
$\sum\limits_{e\in E(G)}\tau(G/e)$ is the size of the following set:
$$
\Phi=\{(e,T): e\in E(T),
T \mbox{ is a spanning tree of } G \mbox{ with } e\in E(T)\}.
$$
For each spanning tree $T$,
$\Phi$ contains exactly $n-1$ elements,
implying that $|\Phi|=\tau(G)\cdot (n-1)$.
Thus
$$
\sum_{ij\in E(G)}\tau(G/ij)=(n-1)\tau(G)
$$
and
$$\sum_{ij\in E(G)} R_{ij}=n-1.$$
{\hfill$\Box$}

It was noted in \cite{Tetali94} that Foster's Theorem can be extended to graphs
with arbitrary resistances.

\begin{theorem}[\cite{Foster, Tetali94}]\label{global-2}
Let $G$ be a connected graph with $n$ vertices. Each edge $ij$
has resistance $r_{ij}$. Then
$$\sum_{ij\in E(G)} \frac{R_{ij}}{r_{ij}}=n-1.$$
\end{theorem}

Theorem \ref{global-2} can also be proved by considering
$R_{ij}=\frac{\tau(G/\{i, j\})}{\tau(G)}$. The proof is almost the
same. We leave it to the readers.

We call Foster's Theorem the global rule for the effective resistance of graphs.
In 2008, Chen and Zhang \cite{CZ} proved a local sum rule for the resistance of graphs.

\begin{theorem}[\cite{CZ}]\label{local}
Let $G$ be a simple connected graph with $n$ vertices.
Each edge is assigned unit resistance. For any two vertices $u$ and $v$ of $G$,
$$d(u)R_{uv}+\sum_{x\in N(u)} (R_{ux}-R_{vx})=2,$$
where $d(u)$ is the degree of vertex $u$ and $N(u)$ is the set of neighbors of $u$ in $G$.
\end{theorem}

Theorem \ref{local} offers a complete set of local rules
that can determine all resistance distances of a graph.
The following two theorems can be viewed as special cases of Theorem \ref{local}.

\begin{theorem}[\cite{Gervacio}]\label{cor1}
Let $G$ be a graph with two non-adjacent vertices $a$ and $b$ such that $N(a)=N(b)$.
Then $R_{a, b}=\frac{2}{|N(a)|}$.
\end{theorem}

\begin{theorem}\label{cor2}
Let $G$ be a graph with two adjacent vertices $a$ and $b$ such that $N(a)=N(b)$.
Then $R_{a, b}=\frac{2}{|N(a)|+1}$.
\end{theorem}

Similar to Foster's Theorem, Chen and Zhang's local rules
can also be extended to graphs with arbitrary resistances.

\begin{theorem}[\cite{Chen}]\label{local-2}
Let $G$ be a simple connected graph with $n$ vertices. Each edge $ij$
has resistance $r_{ij}$. For any two vertices $u$ and $v$ of $G$,
$$\left(\sum_{x\in N(u)} \frac{1}{r_{ux}}\right)R_{uv}+\sum_{x\in N(u)} \frac{1}{r_{ux}}(R_{ux}-R_{vx})=2,$$
where $N(u)$ is the set of neighbors of $u$ in $G$.
\end{theorem}

\subsection{Effective resistances and the number of spanning trees of $G(m,n,p)$
\label{sec5-2}}

Let $G(m,n,p)$ denote the nearly complete bipartite graph
$K_{m,n}-pK_2$ with unit resistance on each edge,
where $p\le \min\{m,n\}$.
Assume that the vertex set of $G(m,n,p)$ is
$\{x_1, x_2, \ldots, x_m\} \cup \{y_1, y_2, \ldots, y_n\}$
and its edge set
$\{x_iy_j: 1\leq i\leq m, 1\leq j\leq n\}\backslash \{x_1y_1, x_2y_2, \ldots, x_py_p\}$.

For convenience, let $[k]=\{1, 2, \ldots, k\}$
if $k\in \mathbb{N}^+$ and $[k]=\emptyset$ if $k=0$.

Now we state our main result of this section as follows.

\begin{theorem}\label{resistance}
Assume that $G(m,n,p)=K_{m,n}-pK_2$ has a
unit resistance on each edge.
Then the following conclusions hold:
\begin{enumerate}
\item[(1)] for $1\le i,j\le m$,
\begin{displaymath}
R_{x_ix_j} = \left\{ \begin{array}{ll}
\frac{2(m-1)}{mn-m-n}, &\textrm{if $i\neq j$ and $i,j\in [p]$}, \\
\frac{2}{n}, & \textrm{if $i\neq j$ and $i,j\in [m]\backslash [p]$}, \\
\frac{2n-1}{n(n-1)}+\frac{p-1}{p(n-1)(mn-m-n)}+\frac{n-p}{pn(n-1)(mn-m-n+p)},
& \textrm{if $i\in [p], j\in [m]\backslash [p]$.}
\end{array} \right.
\end{displaymath}

\item[(2)] for $1\le i,j\le n$,
\begin{displaymath}
R_{y_iy_j} = \left\{ \begin{array}{ll}
\frac{2(n-1)}{mn-m-n}, &\textrm{if $i\neq j$ and $i,j\in [p]$}, \\
\frac{2}{m}, &\textrm{if $i\neq j$ and $i,j\in [n]\backslash [p]$}, \\
\frac{2m-1}{m(m-1)}+\frac{p-1}{p(m-1)(mn-m-n)}+\frac{m-p}{pm(m-1)(mn-m-n+p)},
& \textrm{if $i\in [p], j\in [n]\backslash [p]$.}
\end{array} \right.
\end{displaymath}

\item[(3)] for $1\le i\le m$ and $1\le j\le n$,
\begin{displaymath}
R_{x_iy_j} = \left\{ \begin{array}{ll}
\frac{m+n}{mn-m-n}-\frac{mn}{(mn-m-n)(mn-m-n+p)},
&\textrm{if $i= j\in [p]$}, \\
\frac{m+n-2}{mn-m-n}-\frac{mn}{(mn-m-n)(mn-m-n+p)},
&\textrm{if $i\neq j$ and $i,j\in [p]$}, \\
\frac{1}{m}+\frac{(p-1)(m-1)}{p(mn-m-n)}+\frac{(m-p)(m-1)}{pm(mn-m-n+p)},
& \textrm{if $i\in [p], j\in [n]\backslash [p]$},\\
\frac{1}{n}+\frac{(p-1)(n-1)}{p(mn-m-n)}+\frac{(n-p)(n-1)}{pn(mn-m-n+p)},
& \textrm{if $i\in [m]\backslash [p], j\in [p]$}, \\
\frac{1}{m}+\frac{1}{n}-\frac{mn-m-n}{mn(mn-m-n+p)},
& \textrm{if $i\in [m]\backslash [p], j\in [n]\backslash [p]$.}
\end{array} \right.
\end{displaymath}
\end{enumerate}
\end{theorem}

\begin{proof}
By symmetry of $G(m, n, p)$, we set
\begin{equation*}\begin{split}
& r_1=R_{x_ix_j}, ~\textrm{for $i\neq j$ and $i,j\in [p]$}, \\
& r_2=R_{y_iy_j}, ~\textrm{for $i\neq j$ and $i,j\in [p]$}, \\
& r_3=R_{x_ix_j}, ~\textrm{for $i\neq j$ and $i,j\in [m]\backslash [p]$}, \\
& r_4=R_{y_iy_j}, ~\textrm{for $i\neq j$ and $i,j\in [n]\backslash [p]$}, \\
& r_5=R_{x_iy_i}, ~\textrm{for $i\in [p]$}, \\
& r_6=R_{x_iy_j}, ~\textrm{for $i\neq j$ and $i,j\in [p]$}, \\
& r_7=R_{x_ix_j}, ~\textrm{for $i\in [p], j\in [m]\backslash [p]$}, \\
& r_8=R_{x_iy_j}, ~\textrm{for $i\in [p], j\in [n]\backslash [p]$}, \\
& r_9=R_{x_iy_j}, ~\textrm{for $i\in [m]\backslash [p], j\in [p]$}, \\
& r_{10}=R_{y_iy_j}, ~\textrm{for $i\in [p], j\in [n]\backslash [p]$}, \\
& r_{11}=R_{x_iy_j}, ~\textrm{for $i\in [m]\backslash [p], j\in [n]\backslash [p]$}.
\end{split}\end{equation*}

\setcounter{claim}{0}
\begin{claim}
$r_3=\frac{2}{n}$, $r_4=\frac{2}{m}$.
\end{claim}
\begin{proof}
Claim 1 follows directly from Theorem \ref{cor1}.
\end{proof}

\begin{claim}
$r_1=\frac{2(m-1)}{mn-m-n}$, $r_2=\frac{2(n-1)}{mn-m-n}$.
\end{claim}
\begin{proof}
Applying Theorem \ref{local} to $\{x_1, x_2\}$ we obtain
\begin{align}\label{1}
(n-1)r_1+(r_6-r_5)=2.
\end{align}
Applying Theorem \ref{local} to $\{y_1, y_2\}$ we obtain
\begin{align}\label{2}
(m-1)r_2+(r_6-r_5)=2.
\end{align}
From (\ref{1}) and (\ref{2}) we have
\begin{align}\label{3}
r_1=\frac{m-1}{n-1}r_2.
\end{align}

Applying Theorem \ref{local} to $\{x_1, y_1\}$ we obtain the following
two equations.
\begin{align}\label{4}
(n-1)r_5+(p-1)(r_6-r_2)+(n-p)(r_8-r_{10})=2.
\end{align}
\begin{align}\label{5}
(m-1)r_5+(p-1)(r_6-r_1)+(m-p)(r_9-r_7)=2.
\end{align}
Applying Theorem \ref{local} to $\{x_1, y_2\}$
leads to two equations, we choose the following one of them.
\begin{align}\label{6}
(n-1)r_6+(r_6-0)+(p-2)(r_6-r_2)+(n-p)(r_8-r_{10})=2.
\end{align}
$(\ref{6})-(\ref{4})$ we obtain
\begin{align}\label{7}
(n-1)(r_6-r_5)+r_2=0.
\end{align}
Combine (\ref{2}) and (\ref{7}) we have
$$r_2=\frac{2(n-1)}{mn-m-n}.$$
Then from (\ref{3}) we can see that
$$r_1=\frac{2(m-1)}{mn-m-n}.$$
\end{proof}

Now we turn to the rest effective resistances.
Since we have known the exact value of $r_1$ and $r_2$,
(\ref{1}), (\ref{4}) and (\ref{5}) become
\begin{align}\label{8}
r_5-r_6=\frac{2}{mn-m-n},
\end{align}
\begin{align}\label{9}
(n-1)r_5+(p-1)r_6+(n-p)(r_8-r_{10})=2+\frac{2(n-1)(p-1)}{mn-m-n},
\end{align}
and
\begin{align}\label{10}
(m-1)r_5+(p-1)r_6+(m-p)(r_9-r_7)=2+\frac{2(m-1)(p-1)}{mn-m-n}.
\end{align}

Applying Theorem \ref{local} to $\{x_1, x_{p+1}\}$ we obtain the following
two equations.
$$(n-1)r_7+(p-1)(r_6-r_9)+(n-p)(r_8-r_{11})=2.$$
$$nr_7+(r_9-r_5)+(p-1)(r_9-r_6)+(n-p)(r_{11}-r_8)=2.$$
Add them up we obtain
\begin{align}\label{11}
(2n-1)r_7+r_9-r_5=4.
\end{align}

Applying Theorem \ref{local} to $\{y_1, y_{p+1}\}$ we obtain the following
two equations.
$$(m-1)r_{10}+(p-1)(r_6-r_8)+(m-p)(r_9-r_{11})=2.$$
$$mr_{10}+(r_8-r_5)+(p-1)(r_8-r_6)+(m-p)(r_{11}-r_9)=2.$$
Add them up we obtain
\begin{align}\label{12}
(2m-1)r_{10}+r_8-r_5=4.
\end{align}

Applying Theorem \ref{local} to $\{x_1, y_{p+1}\}$ we obtain
$$(n-1)r_8+(p-1)(r_6-r_{10})+(r_8-0)+(n-p-1)(r_8-r_4)=2,$$
that is,
\begin{align}\label{13}
(p-1)(r_6-r_{10})+(2n-p-1)r_8=\frac{2(m+n-p-1)}{m}.
\end{align}

Foster's Theorem (Theorem (\ref{global})) shows
\begin{align}\label{14}
p(p-1)r_6+p(n-p)r_8+p(m-p)r_9+(m-p)(n-p)r_{11}=m+n-1.
\end{align}

Solving the system of linear equations consists of
Eq. (8)--(14), we get the following results we need:
\begin{displaymath}
\left\{ \begin{array}{ll}
r_5=\frac{m+n}{mn-m-n}-\frac{mn}{(mn-m-n)(mn-m-n+p)}, \\
r_6=\frac{m+n-2}{mn-m-n}-\frac{mn}{(mn-m-n)(mn-m-n+p)}, \\
r_7=\frac{2n-1}{n(n-1)}+\frac{p-1}{p(n-1)(mn-m-n)}+\frac{n-p}{pn(n-1)(mn-m-n+p)}, \\
r_8=\frac{1}{m}+\frac{(p-1)(m-1)}{p(mn-m-n)}+\frac{(m-p)(m-1)}{pm(mn-m-n+p)}, \\
r_9=\frac{1}{n}+\frac{(p-1)(n-1)}{p(mn-m-n)}+\frac{(n-p)(n-1)}{pn(mn-m-n+p)}, \\
r_{10}=\frac{2m-1}{m(m-1)}+\frac{p-1}{p(m-1)(mn-m-n)}+\frac{m-p}{pm(m-1)(mn-m-n+p)}, \\
r_{11}=\frac{1}{m}+\frac{1}{n}-\frac{mn-m-n}{mn(mn-m-n+p)}.
\end{array} \right.
\end{displaymath}
\end{proof}

Now we turn to the number of spanning trees of $G(m,n,p)$ and prove Theorem \ref{nearly}.

\noindent
{\bf Proof of Theorem \ref{nearly}.} Recall Theorem \ref{res-spanning},
we have $\frac{\tau(G(m,n,p)/x_{p+1}y_{p+1})}{\tau(G(m,n,p))}=r_{11}$,
implying that
\begin{eqnarray*}
\ \frac{\tau(G(m,n,p+1))}{\tau(G(m,n,p))}
& = & \frac{\tau(G(m,n,p))-\tau(G(m,n,p)/x_{p+1}y_{p+1})}{\tau(G(m,n,p))} \\
& = & 1-r_{11} \\
& = & 1-\frac{1}{m}-\frac{1}{n}+\frac{mn-m-n}{mn(mn-m-n+p)} \\
& = & \frac{(mn-m-n)(mn-m-n+p+1)}{mn(mn-m-n+p)}.
\end{eqnarray*}

Therefore,
\begin{eqnarray*}
\ \tau(G(m,n,p))
& = & \left(\prod_{k=0}^{p-1} \frac{\tau(G(m,n,k+1))}{\tau(G(m,n,k))}\right) \cdot\tau(G(m,n,0))\\
& = & \frac{(mn-m-n)^p(mn-m-n+p)}{m^pn^p(mn-m-n)}\cdot\tau(K_{m,n}) \\
& = & (mn-m-n+p)(mn-m-n)^{p-1}m^{n-p-1}n^{m-p-1}.
\end{eqnarray*}
{\hfill$\Box$}

\subsection{Kirchhoff indices of $G(m,n,p)$
\label{sec5-3}}

In 2016, Shi and Chen obtained the Kirchhoff index of $G(n,n,p)=K_{n,n}-pK_2$ as follows.

\begin{theorem}[\cite{SC}]\label{Shi-Chen}
Let $G$ be a graph constructed by removing $p$ disjoint edges
from the complete bipartite graph $K_{n,n} (p\leq n)$. Then
$$Kf(G)=\frac{np(2n^2-5n+2p)}{(n-2)(n^2-2n+p)}+\frac{(n-p)(2n^2-5n+2p+2)}{n^2-2n+p}+2(n-1).$$
Particularly, if $p = n$, we have
$$Kf(G)=\frac{5n-6}{(n-1)(n-2)}+4n+1.$$
\end{theorem}

Since we have obtained effective resistances for $G(m,n,p)$,
the Kirchhoff index of $G(m,n,p)$ is an obvious corollary of Theorem \ref{resistance},
which extends Shi and Chen's results.

\begin{corollary}\label{Kirchhoff}
\begin{eqnarray*}
\ Kf(G(m,n,p)) & = & m+n-1+\frac{p^2(m+n-2)+2p}{mn-m-n}+\frac{(m-p)(m-1)}{n}+\frac{(n-p)(n-1)}{m} \\
&   & -\frac{pmn}{(mn-m-n)(mn-m-n+p)}+\frac{p(m-p)}{n-1}+\frac{(m-p)(p-1)}{(n-1)(mn-m-n)} \\
&   & +\frac{(m-p)(n-p)}{n(n-1)(mn-m-n+p)}+\frac{p(n-p)}{m-1}+\frac{(n-p)(p-1)}{(m-1)(mn-m-n)} \\
&   & +\frac{(m-p)(n-p)}{m(m-1)(mn-m-n+p)}.
\end{eqnarray*}
\end{corollary}

\begin{proof}
\begin{eqnarray*}
\ Kf(G(m,n,p)) & = & \sum_{\{u, v\}: u, v\in V(G)} R_{uv} \\
& = & \sum_{\{u, v\}: uv\in E(G)} R_{uv}+\sum_{\{u, v\}: uv\notin E(G)} R_{uv} \\
& = & m+n-1+\binom{p}{2}r_1+\binom{p}{2}r_2+\binom{m-p}{2}r_3+\binom{n-p}{2}r_4+pr_5 \\
&   & +p(m-p)r_7+p(n-p)r_{10} \\
& = & m+n-1+\frac{p^2(m+n-2)+2p}{mn-m-n}+\frac{(m-p)(m-1)}{n}+\frac{(n-p)(n-1)}{m} \\
&   & -\frac{pmn}{(mn-m-n)(mn-m-n+p)}+\frac{p(m-p)}{n-1}+\frac{(m-p)(p-1)}{(n-1)(mn-m-n)} \\
&   & +\frac{(m-p)(n-p)}{n(n-1)(mn-m-n+p)}+\frac{p(n-p)}{m-1}+\frac{(n-p)(p-1)}{(m-1)(mn-m-n)} \\
&   & +\frac{(m-p)(n-p)}{m(m-1)(mn-m-n+p)}.
\end{eqnarray*}
\end{proof}

\section*{Acknowledgements}

Jun Ge is partially supported by NSFC (No. 11701401) and China Scholarship Council (No. 201708515087).

\end{document}